\documentclass[11pt]{amsart}

\def\da{\underline a}
\def\ga{\overline a}
\newtheorem{theorem}{Theorem}
\newtheorem{proposition}[theorem]{Proposition}
\newtheorem{lemma}[theorem]{Lemma}

\theoremstyle{remark}
\newtheorem{remark}[theorem]{\bf Remark}
\newtheorem{example}[theorem]{\bf Example}
\usepackage{epsfig} %For pictures: screened artwork should be set up with an 85 or 100 line screen
\usepackage{graphicx}
\usepackage{enumerate}

\begin{document}

\title{Asymptotic properties of a general model of immune status} 
\author[K. Pich\'or]{Katarzyna Pich\'or}
\address{K. Pich\'or, Institute of Mathematics,
University of Silesia, Bankowa 14, 40-007 Kato\-wi\-ce, Poland.}
\email{katarzyna.pichor@us.edu.pl}
\author[R. Rudnicki]{Ryszard Rudnicki}
\address{R. Rudnicki, Institute of Mathematics,
Polish Academy of Sciences, Bankowa 14, 40-007 Katowice, Poland.}
\email{rudnicki@us.edu.pl}
\keywords{Immune status, physiologically structured population, stochastic semigroup, asymptotic stability, flow with jumps}
\subjclass[2020]{47D06; 35Q92; 60J76; 92D30}
%\date{May 30th, 2021}

\thanks{This research was partially supported by
the  National Science Centre (Poland)
Grant No. 2017/27/B/ST1/00100}

\begin{abstract}
We consider a model of  dynamics of the immune system.
The model is based on three factors: 
occasional boosting and continuous waning of immunity and 
a general description of the period between subsequent boosting events.
The antibody concentration changes according to a non-Markovian process.   
The density of the distribution of this concentration satisfies some
partial differential equation with an integral boundary condition.
We check that this system generates a stochastic semigroup and we study the long-time behaviour of this semigroup.
In particular we prove a theorem on its asymptotic stability.
\end{abstract}

\maketitle

%\footnotetext{This research was partially supported by
%the  National Science Centre (Poland) Grant No. 2017/27/B/ST1/00100.}
%\footnotetext{\textbf{Abbreviations:} ANA, anti-nuclear antibodies; APC, %antigen-presenting cells; IRF, interferon regulatory factor}

\section{Introduction}
\label{intro}
Most of epidemiological models belong to two main different groups \cite{Keeling-book}. 
One group consists of models in which the infected individuals 
become fully immune after recovery, and this protection lasts for life (SIR, SEIR, MSIR models).
The second group includes models in which those cured are immediately susceptible
to infection with the same risk as before (SIS, SEIS, SIRS models).
However there are many diseases in which infection-derived immunity is imperfect and a recovered individual is only
temporary resistant. Models which include temporary immunity (e.g. SIRS, SEIRS) are more advanced 
because the transition between the different epidemiological compartments is quite complicated and depends on the period of immunity, 
which has usually a random length \cite{Blackwood,Dushoff,Le}.
It is worth mentioning
that  naturally acquired  immunity can differ significantly
from vaccine-induced immunity \cite{Wendelboe} (in the case of \textit{Bordetella pertussis}).
 
Since the dynamics of the immune system is an essential component of 
epidemiological models with temporary immunity, it is worth exploring models that describe the immune status \cite{DGKT}.
The \textit{immune status} is the concentration of specific antibodies, which appear after infection with a pathogen
and remain in serum, providing protection against future attacks of the same pathogen. Over time, 
the number of  antibodies decreases until the next infection.  
For example, immunity acquired through infection against \textit{dengue virus} wanes to a level that allows
subsequent infection after an average of two years \cite{Reich},
but for pertussis wanes after 4-20 years \cite{Lavine-Broutin,Versteegh,Wendelboe}
or even after 30 years \cite{Wearing}.
In the case of dengue virus infection, antibody concentrations during primary infection may increase 
to a level sufficient  for inhibition of secondary infection.
 However, immunity wanes  over time,
allowing for secondary infection by serotypes different from the primary strain \cite{Gulbudak}.

During an infection, the immunity is boosted and then 
the immunity is gradually waning, etc. 
Thus, the antibody concentration
is described by a stochastic process whose 
trajectories are decreasing functions $x(t)$ between subsequent infections.
These functions satisfy the differential equation
\begin{equation}
\label{wane}
x'(t)=F(x(t)).
\end{equation}
For example, in \cite{Antia} it was supposed that the antibody level declines exponentially, i.e. $F(x)=-cx$, $c>0$.
We assume that the time it takes the immune system to clear infection is negligible
and that if $x$ is the concentration of antibodies at the moment of infection, then $G(x)>x$ is the concentration of antibodies 
just after clearance of infection. 
An explicit expression for $G$ was derived in \cite{GKTD,TEGB-MK}.
 
In \cite{DGKT} the authors consider functions $G$ which are unimodal and have properties: $\lim_{x\to0}G(x)=\infty$
and $\lim_{x\to\infty}(G(x)-x)=\text{const}$. They proved the existence of an asymptotically stable  stationary density of the process $x(t)$.
The result from the paper \cite{DGKT} was extended  
in \cite{PR-immunology} to a large class of function $G$ 
including a significant case when the increase of the concentration of antibodies after the infection is bounded.

In \cite{DGKT,PR-immunology} it was assumed that the moments of infections are independent of the state of the immune system and they
are distributed according to a Poisson process $(N_t)_{t\ge 0}$. It means that the period between 
infections is exponentially distributed. In this paper we consider a more general model where the length of this period depends on the immune status $x_b$
just after clearance of infection
and has the probability density distribution $a\mapsto q(x_b,a)$, i.e. 
the integral $\int_0^a q(x_b,s)\, ds$ is the probability that this period has the length $\tau\le a$.
The form of the function $q$ depends on many factors, including the infectious disease under consideration, the degree of population immunity, and the preventive measures taken \cite{Metcalf}. Certain infections occur seasonally, most of them in annual cycles but, for example,
there is statistical evidence for three or four year cycles of pertussis \cite{Broutin,Lavine-short}.
In the case of seasonal diseases  the function $q$ has some local maxima, thus the exponential function does not describe 
the distribution of the period between infections and
it is better to assume that $q$ has a gamma distribution with respect to $a$ \cite{Wearing}. 
The function $q$ can also depend on $x_b$ because the re-infection occurs when antibody level falls below a certain threshold value~\cite{Versteegh}.
   
The immune status is a flow on the interval $[0,\infty)$ with jumps at random moments $t_0<t_1<t_2<t_3<\dots$.
We denote this process by $(\xi_t)_{t\ge 0}$ and it is defined by the following equations
\[
\xi_{t_n}=G(\xi_{t_{n}^-}),\quad \xi_t'=F(\xi_t)\,\,\,\textrm{for $t\in [t_{n},t_{n+1})$},\,\ t_{n+1}-t_{n}=\tau_n,\  n=0,1,2,\dots
\]
and the random variables 
$\tau_0=t_1-t_0$, $\tau_1=t_2-t_1$, $\tau_2=t_3-t_2,\dots$,
are independent and $\tau_n$ has the density distribution $q(\xi_{t_n},a)$.
Then $(\xi_t)_{t\ge 0}$ is a piecewise deterministic process \cite{davis84,RT-K-k}.
Let $f(t)$ be  the density of the random variable $\xi_t$.
Our main goal is to prove that under some general assumptions there exists a unique density $\tilde g_*\in L^1[0,\infty)$ such that 
\[
\lim_{t\to\infty} \|f(t)-\tilde g_*\|=0
\]
(see Section~\ref{s:as-one-dim}).

The main difficulty in proving this result is that the process $(\xi_t)_{t\ge 0}$ is
only Markovian when the random variables $\tau_n$ have an exponential distribution.
In our case, the function $f$ is not a solution of some evolution equation 
$f'(t)=\mathcal Af(t)$ on  the space $L^1[0,\infty)$. Thus the process $(\xi_t)_{t\ge 0}$ does not generate a semigroup of operators  
and we cannot use methods directly from the papers \cite{DGKT,PR-immunology} to study the behaviour of $f(t)$.

To overcome these obstacles we introduce a  Markov process $(\eta_t)_{t\ge 0}$ 
given by the formula
\[
\eta_t=(\xi_{t_n},t-t_n) \quad\text{for $t_n\le t< t_{n+1}$.}
\]
Thus, the process $(\eta_t)_{t\ge 0}$ describes the immune status
at the  moment of the last jump and the time that has elapsed since that jump.
This process takes values in the space $[0,\infty)^2$
and has Markov property with respect to its natural filtration 
$\mathcal F_{\le t}=\sigma(\eta_s\colon s\le t)$, $t\ge 0$.
Under suitable assumptions, the random variables $\eta_t$ have densities if 
the random variable $\eta_0$ has a density.
Accordingly, the process $(\eta_t)_{t\ge 0}$ generates a semigroup of operators, convenient for studying the behavior of its densities and, consequently, also the asymptotics of $f(t)$.
   
Let $a=t-t_n$, $x_b=\xi_{t_n}$, and we denote by $u(t,x_b,a)$
the density of the distribution of $\eta_t$.  If $v(x_b,a)=u(0,x_b,a)$ then we define $P(t)v(x_b,a)=u(t,x_b,a)$.  
We will check whether $\{P(t)\}_{t\ge 0}$ is a stochastic semigroup on some $L^1$ space
and we will describe the long-time behaviour of this semigroup. 
We are mainly interested in the asymptotic stability of this semigroup. 
Finally, we return to the distribution of the process $(\xi_t)_{t\ge 0}$
and we deduce from the properties of the semigroup $\{P(t)\}_{t\ge 0}$, the asymptotic behaviour of the densities $f(t)$. 
In particular from asymptotic stability of the semigroup  $\{P(t)\}_{t\ge 0}$ it follows that 
the process $(\xi_t)_{t\ge 0}$ has a unique stationary density $\tilde g_*$ and, according to the ergodic theorem, 
$\tilde g_*$ is the density of the distribution of the immune status in the population.  
The main idea of the paper is to formulate the problem in the terms of stochastic semigroups and then apply some results concerning
the Foguel alternative~\cite{PR-JMMA2016,PR-SD2017}, which gives conditions when a stochastic semigroup is asymptotically stable or sweeping.

The organization of the paper is as follows.
In Section~\ref{s:model} we present the assumptions concerning our model and 
we introduce a system of equations for $u(t,x_b,a)$.
The system consists of a partial differential equation of the first order, an integral type boundary condition 
and an initial condition.
In Sections~\ref{s:semigroup} and~\ref{s:restr-semigroup} we introduce a stochastic semigroup related to this system.
Section~\ref{s:gen-asympt} contains general definitions and results concerning long-time behaviour of stochastic semigroups.
The main result of the paper is asymptotic stability of the semigroup $\{P(t)\}_{t\ge 0}$ and its
proof  is given in the next three sections.
In Section~\ref{s:sweeping} we study the sweeping property which can be interpreted as permanent immunity of the population.
In Section~\ref{s:as-one-dim} we apply results concerning asymptotic stability of the semigroup  $\{P(t)\}_{t\ge 0}$ to the process $(\xi_t)$.  
We show  that the semigroup generated by the process $\zeta_t=(\xi_t,t-t_n)$, $t_n\le t< t_{n+1}$,
is also asymptotically stable, which implies convergence of the marginal distributions of $(\xi_t)_{t\ge 0}$ to some stationary density.   
In the last section we analyze  versions of the model with specific  functions $F$, $G$ and $q$, e.g. 
when immunity decreases exponentially; with constant increase of antibodies after infection; with a threshold concentration of antibodies at the re-infection;
and with seasonal infections.

\section{Description of the model}
\label{s:model}
We recall that the immunity $x$ increases during infection to the value $G(x)$ and then decreases   
according to the equation $x'(t)=F(x(t))$. We also recall that if $x_b$ is the immunity after an infection, then 
the function $a\mapsto q(x_b,a)$ is the distribution density of the time to the next infection. 
We assume that the functions $F$, $G$, and $q$ satisfy the following conditions: 
\begin{itemize}
\item[(A1)] $F\colon [0,\infty)\to \mathbb R$ is a $C^1$-function such that $F(x)<0$ for $x>0$ and $F(0)=0$,
\item[(A2)] $G\colon [0,\infty)\to (0,\infty)$ is a $C^1$-function such that $G(x)>x$ for $x\ge 0$, 
\item[(A3)] there exists an at most countable family of pairwise disjoint  open intervals $\Delta_i$, $i\in I$, such that 
the set  $[0,\infty)\setminus\bigcup_{i\in I}\Delta_i$  
has Lebesgue measure zero
and $G'(x)\ne 0$ for $x\in \Delta_i$ and $i\in I$,
\item[(A4)] $q\colon [0,\infty)\times [0,\infty)\to [0,\infty)$ is a  continuous function and 
for each $x_b\in [0,\infty)$ the function $a\mapsto q(x_b,a)$ is a probability density,
\item[(A5)] there exists $\varepsilon>0$ such that
$\int_0^{\varepsilon} q(x_b,a)\,da<1-\varepsilon $ for all $x_b\ge 0$.
\end{itemize}
%We recall that $\mathbf I=[x_{\rm min},\infty)$,  where $x_{\rm min}=\min\{G(x)\colon x\ge 0\}$.

We denote by $\pi_tx_0$ the solution $x(t)$ of Eq.~\eqref{wane} with the initial condition $x(0)=x_0\ge 0$.
We define $\ga(x_b)\le \infty$ to be the minimum number such that
\[
\int_0^{\ga(x_b)}q(x_b,a)\,da=1\quad\textrm{for $x_b\ge 0$}.
\]
We assume that  $\ga(x_b)=\infty$ if  $\int_0^r q(x_b,a)\,da<1$  for all $r>0$. Let
\[
Y=\{(x_b,a)\colon    x_b\ge 0,\,\,\, 0\le a<\ga(x_b)\}.
\]
The process $(\eta_t)_{t\ge 0}$ has values in  $Y$.

Assumption (A3) allows us to introduce   \cite{Rudnicki-LN,RT-K-k} a linear operator $P_G$ on the space $L^1=L^1[0,\infty)$ given by the formula
\begin{equation}
\label{F-P-operator-gladki}
P_Gf(x)=\sum_{i\in I_x} f(\varphi_i(x))|\varphi_i'(x)|,
\end{equation} 
where $\varphi_i$ is the inverse function of 
$G\big|_{\Delta_i}$
and $I_x=\{i\colon x\in G(\Delta_i)\}$.
Then $P_G$ is a {\it Frobenius--Perron operator} \cite{LiM} for the transformation~$G$, i.e. 
$P_G$ satisfies the following condition 
\begin{equation}
\label{def-FP}
\int_A P_G f(x)\,dx=
\int_{G^{-1}(A)} f(x)\,dx
\end{equation}
for each $f\in L^1$ and all Borel subsets $A$ of $[0,\infty)$.

The operator $P_G^*\colon L^{\infty}[0,\infty)\to L^{\infty}[0,\infty)$
adjoint of the Frobenius--Perron operator $P_G$
is given by $P_G^*f(x)= f(G(x))$
and it is called the \textit{Koopman operator} or the 
\textit{composition operator}.

Denote by $D$ the subset of the space
$L^1$ which contains all
\textit{densities}
\[
D=\{f\in L^1\colon \,\, f\ge 0,\,\, \|f\|=1\}.
\]
The Frobenius--Perron operator describes the evolution of densities under the action of the transformation $G$ and it is an example of a
\textit{stochastic} or \textit{Markov  operator},
which is defined as a linear operator $P\colon  L^1\to L^1$ 
such that $P(D)\subset D$. We  also use the notion of a \textit{stochastic semigroup}, which is 
a $C_0$-semigroup of stochastic operators.  
We recall that a family $\{U(t)\}_{t\ge0}$ of linear
operators on a Banach space $E$
is a $C_0$-\textit{semigroup}
or \textit{strongly continuous semigroup}
if it satisfies the following conditions:
%\begin{itemize}[\rm(a)]
\begin{enumerate}[\rm(a)]
\item \ $U(0)=I$,  i.e., $U(0)f =f$ for $f\in E$,
\item \ $U(t+s)=U(t) U(s)\quad \textrm{for}\quad
s,\,t\ge0$,
\item \  for each $f\in E$ the function
$t\mapsto U(t)f$ is continuous.
\end{enumerate}
%\end{itemize}

We now define a stochastic semigroup related to the process $(\eta_t)_{t\ge 0}$. 
We recall that $\eta_t=(\xi_{t_n},t-t_n)$ for $t\in [t_n,t_{n+1})$, where 
$t_0<t_1<t_2<\dots$ are jump times for the process $(\xi_t)_{t\ge 0}$.
We have $\xi_{t_{n+1}}=G_{\tau_n}(\xi_{t_n})$, where $\tau_n=t_{n+1}-t_n$ 
and $G_a(x_b)=G(\pi_ax_b)$
for $x_b\ge 0$ and $a\ge 0$.
For each $a\ge 0$ the transformation $G_a$ is a $C^1$-function and
satisfies (A3).  We denote by $P_a$ the Frobenius--Perron operator corresponding to $G_a$. Then  
$P_a=P_GP_{\pi_a}$ and according to
\eqref{F-P-operator-gladki} we have  
\begin{equation}
\label{F-P-pi}
P_{\pi_a}f(x)=f(\pi_{-a}x)\frac{F(\pi_{-a}x)}
{F(x)}
\end{equation}
if $\pi_{-a}x$ exists and $P_{\pi_a}f(x)=0$ otherwise.
The adjoint of the operator $P_a$ is given by
$P_a^*f(x)=f(G(\pi_ax))$.

If the process $(\xi_t)_{t\ge 0}$ starts from the point $x_b$,
then the jump rate at time $a$ is given by
\[
p(x_b,a)=\lim_{\Delta t\downarrow 0}
\frac{\operatorname P(\tau\in [a,a+\Delta t] \mid \tau\ge a)}{\Delta t},
\]
where the random variable $\tau$ is the length of the period between jumps. 
Let $\Phi(x_b,a)=\int_a^{\infty} q(x_b,r)\,dr$. An easy computation shows that
$\Phi(x_b,a)=\exp\big(-\int_0^a p(x_b,r)\,dr\big)$, which gives 
\begin{align}
\label{function p1}
q(x_b,a)&=p(x_b,a)\exp\big(-\textstyle{\int_0^a} p(x_b,r)\,dr\big),\\
\label{function p2}
\quad p(x_b,a)&=\frac{q(x_b,a)}{\int_a^{\infty} q(x_b,r)\,dr} 
\end{align}
 for $a<\ga(x_b)$ and we set $ p(x_b,a)=0$ for $a\ge \ga(x_b)$. 
As   $\Phi(x_b,\ga(x_b))=0$, we have 
\begin{equation}
\label{function p3}
\int_r^{\ga(x_b)}p(x_b,a)\,da=\infty \quad\textrm{for $r<\ga(x_b)$}. 
\end{equation}

Let $u(t,x_b,a)$ be  the density of  distribution of $\eta_t$.
Since $p(x_b,a)$ is the rate of jump 
of the process $(\eta_t)_{t\ge 0}$ 
from $(x_b,a)$ to $(G_a(x_b),0)$ and since $P_a$ is the Frobenius--Perron operator corresponding to $G_a$, 
we have 
\begin{equation}
\label{boubd-cond}
u(t,x_b,0)=\int_0^{\infty}
\Big(P_a\big(p(\cdot,a)u(t,\cdot,a)\big)\Big)(x_b)\,da . 
\end{equation}
Though we consider $a\le\ga(x_b)$, it will be convenient to keep in the paper the notation of integral  $\int_0^{\infty}$ with respect to $a$ as in 
formula (\ref{boubd-cond}) assuming that $u(t,x_b,a)=0$ for $a> \ga(x_b)$.
We will use the shortened notation $\mathcal Pu(t,x_b)$
for the expression on the right-hand side of (\ref{boubd-cond}).
Thus, equation (\ref{boubd-cond}) takes the form $u(t,x_b,0)=\mathcal Pu(t,x_b)$. 
We will also write $\mathcal Pf(x_b)$ instead of $\int_0^{\infty}
\Big(P_a\big(p(\cdot,a)f(\cdot,a)\big)\Big)(x_b)\,da$.

Hence the function $u$ satisfies the following initial-boundary problem:
\begin{align}
\label{eq1}
&\frac{\partial u}{\partial t}(t,x_b,a)
 +\frac{\partial u}{\partial a}(t,x_b,a)
 =-p(x_b,a)u(t,x_b,a),\\
&u(t,x_b,0)=\mathcal Pu(t,x_b), 
\label{eq2}\\
&u(0,x_b,a)=u_0(x_b,a). 
\label{eq3}
\end{align}

\section{Stochastic semigroup}
\label{s:semigroup}

We show that system (\ref{eq1})--(\ref{eq3}) generates a stochastic semigroup on the space $E=L^1(Y,\mathcal B(Y),m)$, where $\mathcal B(Y)$ is the $\sigma$-algebra of Borel subsets of $Y$ and $m$ is the Lebesgue measure.
Let $\mathcal A$ be an operator with domain
\begin{equation*}
%\label{domain-A}
\mathcal D(\mathcal A)=\Big\{f\in E\colon\,\, \frac{\partial f}{\partial a}\in E,
\,\,pf\in E,\,\, f(x_b,0)=\mathcal Pf(x_b) \Big\}
 \end{equation*}
 given by  
\[
\mathcal Af=-\frac{\partial f}{\partial a}-pf.
\]

Since a function $f\in E$ is only almost everywhere defined, 
the formula for $f(x_b,0)$ needs clarification.
The domain $\mathcal D(\mathcal A)$ is a subset of 
the Sobolev space 
\[
W_1(Y)=\Big\{f\in E\colon \frac{\partial f}{\partial a}\in E\Big\}
\]   
with the norm $\|f\|_{W_1(Y)}=\|f\|_E+\Big\|\dfrac{\partial f}{\partial a}\Big\|_E$. 
In the space $W_1(Y)$ we introduce the trace operator $\mathcal T\colon W_1(Y)\to L^1[0,\infty)$,  $\mathcal Tf(x_b)=f(x_b,0)$  
in the following way (see also \cite{Evans} Chapter 5.5). Let $C^1_c(Y)$ be the space of $C^1$-functions from $Y$ to $\mathbb R$  with compact supports.
Then for $f\in C^1_c(Y)$ we have 
\[
\int_0^{\infty}|f(x_b,0)|\,dx_b\le \iint\limits_Y\Big|\frac{\partial f}{\partial a}(x_b,a)\Big|\,dx_b\,da\le \|f\|_{W_1(Y)}.
\]
Since the set $C^1_c(Y)$ is dense in $W_1(Y)$ we can extend $\mathcal T$ uniquely  to a linear bounded operator
on the whole space $W_1(Y)$.

\begin{theorem}
\label{th:gen}
The operator $\mathcal A$ generates a stochastic semigroup $\{P(t)\}_{t\ge0}$
on~$E$.
\end{theorem}
The proof of this result can be done by using the Hille--Yosida theorem, but it is easer to apply a perturbation method related to
operators with boundary conditions developed in \cite{greiner}
and an extension of this method to unbounded perturbations in $L^1$ space
in \cite{GMTK}. Theorem~\ref{MTK}  below is a version of \cite[Theorem 1]{GMTK}
for stochastic semigroups (see \cite[Remark 2]{GMTK}).

\begin{theorem}
\label{MTK}
Let $(\Gamma,\Sigma,m)$, $(\Gamma_{\partial},\Sigma_{\partial},m_{\partial})$ be $\sigma$-finite measure spaces
and let $L^1=L^1(\Gamma,\Sigma,m)$ and $L_{\partial}^1=L^1(\Gamma_{\partial},\Sigma_{\partial},m_{\partial})$.
Let $\mathcal{D}$ be a linear subspace of~$L^1$.
We assume that
$A\colon \mathcal D\to L^1$ and $\Psi_0,\Psi\colon \mathcal D\to L_{\partial}^1$ are linear operators satisfying the following conditions: 
\begin{enumerate}[\rm(1)]
\item
for each $\lambda>0$, the operator $\Psi_0\colon \mathcal{D}\to L^1_{\partial}$ restricted to the nullspace 
$\mathcal N(\lambda I-A)=\{f\in\mathcal D\colon \lambda f- Af=0\}$ has a  positive right inverse $\Psi(\lambda)\colon L^1_{\partial}\to \mathcal N(\lambda I-A)$, i.e.  
$\Psi_0\Psi(\lambda)f_{\partial}=f_{\partial}$ for $f_{\partial} \in L^1_{\partial}$; 
\item
 the operator $\Psi\colon \mathcal{D} \to L^1_{\partial}$ is positive and there is $\omega>0$ such that $\|\Psi \Psi(\lambda)\|<1$ for $\lambda>\omega$;
\item the operator  $A_0=A\big|_{\mathcal D(A_0)} $, where $\mathcal{D}(A_0)=\{f\in \mathcal{D}\colon  \Psi_0 f=0\}$,
 generates a positive $C_0$-semigroup on $L^1$;
\item
$\int_{\Gamma} Af(x)\,m(dx)\le \int_{\Gamma_{\partial}} \Psi_0 f(x_\partial)\,m_{\partial}(dx_\partial)$
for $f\in \mathcal{D}_+=\{f\in \mathcal{D}\colon f\ge 0\}$;
\item
$\int_{\Gamma} Af(x)\,m(dx)= 0$ for  
$f\in \mathcal{D}(A)=\{f\in \mathcal{D}\colon  \Psi_0 f=\Psi f\}$ and 
$f\ge 0$.
\end{enumerate}
Then the operator $A$ with the domain $\mathcal{D}(A)$ generates a stochastic semigroup.
\end{theorem}

\begin{proof}[Proof of Theorem~$\ref{th:gen}$]
First we translate our notation to that from Theorem~\ref{MTK}.
Let $L^1=E$,  $L_{\partial}^1=L^1[0,\infty)$, 
$\Gamma=Y$, $\Gamma_{\partial}=[0,\infty)$,
$\Psi_0=\mathcal T$, $\Psi =\mathcal P$, $Af=-\frac{\partial f}{\partial a}-pf $ and
\[
\mathcal D=\Big\{f\in E\colon\,\, \frac{\partial f}{\partial a}\in E, \,\,pf\in E\}.
\]

(1): Since $Af=-\frac{\partial f}{\partial a}-pf$, the nullspace  $\mathcal N(\lambda I-A)$ is the set of functions $f\in \mathcal D$ satisfying equation
 \[
\frac{\partial f}{\partial a} +pf+\lambda f=0.
\]
Solving this equation we obtain that $f(x_b,a)=f(x_b,0)\Phi_{\lambda}(x_b,a)$, where
\[
\Phi_{\lambda}(x_b,a)=e^{-\lambda a}\Phi(x_b,a)=\exp\bigg\{-\int_0^a(\lambda+p(x_b,s))\,ds\bigg\}.
\]
Thus the operator $\Psi_0$ restricted to $\mathcal N(\lambda I-A)$ is invertible 
and the inverse operator $\Psi(\lambda)\colon L^1[0,\infty) \to \mathcal N(\lambda I-A)$ given by 
$\Psi(\lambda)f(x_b,a) =f(x_b)\Phi_{\lambda}(x_b,a)$ is positive.

(2): Since $\Psi=\mathcal P$  we  check whether 	$\|\mathcal P\Psi(\lambda)\|<1$ for $\lambda>0$.
Take $f\in L^1[0,\infty)$, $f\ge 0$,  and  let  $\Theta_{\lambda}(x_b,a)=p(x_b,a)\Phi_{\lambda}(x_b,a)=e^{-\lambda a}q(x_b,a)$.
Then 
\begin{align*}
\int_0^{\infty}(\mathcal P\Psi(\lambda)f)(x_b)\,dx_b&=\int_0^{\infty}\int_0^{\infty} P_a (f(\cdot)\Theta_{\lambda}(\cdot,a))(x_b)\,da\,dx_b\\
&= \int_0^{\infty}\int_0^{\infty}f(x_b)\Theta_{\lambda}(x_b,a)\,da\,dx_b.
\end{align*}
Thus, we need to estimate the integral $\int_0^{\infty}\Theta_{\lambda}(x_b,a)\,da$. 
From (A5) it follows that
\begin{align*}
\int_0^{\infty}\Theta_{\lambda}(x_b,a)\,da&=\int_0^{\infty}e^{-\lambda a}q(x_b,a)\,da\\
&\le \int_0^{\varepsilon} q(x_b,a)\,da+ 
\int_{\varepsilon}^{\infty} e^{-\lambda \varepsilon}q(x_b,a)\,da\\
&\le 1-\varepsilon\big( 1-e^{-\lambda\varepsilon}\big) <1.
\end{align*}

(3): The operator $A_0$ generates a positive $C_0$-semigroup $\{P_0(t)\}_{t\ge 0}$ on $E$ given by 
\[
P_0(t)f(x_b,a)=
\begin{cases}
f(x_b,a-t)\exp\big\{-\int_{a-t}^a p(x_b,s)\,ds\big\}
\quad\textrm{for $a> t$,}\\
0 \quad\textrm{for $a<t$}.
\end{cases}
\]

(4): If $f\in \mathcal{D}_+$, then
\begin{align*}
\int_Y Af(x_b,a)\,dx_b\,da
&=-\int_Y\bigg( \frac{\partial f}{\partial a}(x_b,a)+p(x_b,a)f(x_b,a)\bigg)  \,da \,dx_b\\
&=\int_0^{\infty}\mathcal Tf(x_b)\,dx_b-\int_Y p(x_b,a)f(x_b,a)  \,da \,dx_b.
\end{align*}
Since $\Psi_0=\mathcal T$, $\Gamma=Y$, and $\Gamma_{\partial}=[0,\infty)$ we have 
\[
\int_Y Af(x_b,a)\,dx_b\,da-\int_0^{\infty} \Psi_0 f(x_b,a)\,dx_b=
-\int_Y p(x_b,a)f(x_b,a)  \,da \,dx_b\le 0.
\]

(5): If $f\in \mathcal{D}(A)$ and $f$ is nonnegative, then
\begin{align*}
\int_Y Af(x_b,a)\,dx_b\,da&=\int_0^{\infty}\mathcal Tf(x_b)\,dx_b-\int_Y p(x_b,a)f(x_b,a)  \,da \,dx_b\\
&=\int_0^{\infty}\mathcal Pf(x_b)\,dx_b-\int_Y p(x_b,a)f(x_b,a)  \,da \,dx_b=0.\qedhere
\end{align*}

\end{proof}

%\begin{remark}
%In the proof of Theorem~\ref{th:gen} we does not use the assumption that the function $q$ is continuous. 
%To formulate our model it is sufficient to assume that $q$ is measurable, but we will use continuity of $q$ in Section~\ref{s:condition-K}.
%\end{remark}

\section{Restriction of the semigroup}
 \label{s:restr-semigroup}
We now propose some restriction of the semigroup $\{P(t)\}_{t\ge0}$ to a set of densities related to our model.
We add further assumptions concerning the function $q$, which we need to define this restriction and to prove results about asymptotic properties of our model:
\begin{itemize}
\item[(B1)] there exists a continuous function $\da\colon [0,\infty)\to [0,\infty)$
such that $q(x_b,a)>0$ for $a>\da(x_b)$ and 
$q(x_b,a)=0$ for $a<\da(x_b)$ if $\da(x_b)>0$,
\item[(B2)] there exists a constant $M_1>0$ such that for each $x_b$ we have
 \[
 \int_0^{\infty} aq(x_b,a)\,da\le M_1.
 \]
\end{itemize}
From assumption that $q(x_b,a)>0$ for $a>\da(x_b)$ it follows immediately that $\ga(x_b)=\infty$ for each $x_b\ge 0$, and consequently
$Y=[0,\infty)^2$. 

Next we restrict the semigroup $\{P(t)\}_{t\ge0}$ to some space $L^1(X)$, where $X=\mathbf I\times [0,\infty)$ and $\mathbf I\subseteq [0,\infty)$ is an interval   
determined by our model. We assume that the initial immune status $x$ is zero. Then after an infection the immunity is boosted and $x$ becomes $G(0)$ and over time the immune status
decreases. Since $q(x_b,a)>0$  for $a\in (\da(x_b),\infty)$  (or for  $a\in [0,\infty)$ if $q(x_b,0)>0$),  the immune status  after the next infection is 
a number from the non-degenerate interval $ I_1=\{G(\pi_aG(0))\colon q(G(0),a)>0\}$ and $G(0)\in \overline  I_1$, where $\overline  I_1$ denotes the closure of the set $ I_1$.    
We define by induction the sequence of intervals 
\[
 I_{n+1}=\{G(\pi_ax_b)\colon \, x_b\in  I_n,\,\,\,q(x_b,a)>0\}. 
\]
Then $\overline  I_n\subseteq \overline  I_{n+1}$ for each $n\ge 1$. Let $\mathbf I$ be the closure of the set $\bigcup_{n=1}^{\infty}  I_n$. 
In order to characterize some properties of the set $\mathbf I$
 we introduce the notion of cumulative flow.
Let $n \ge 1$ and $\mathbf t=(t_0,t_1,\dots,t_{n-1})$ be  such that  $t_p>0$
for $p=0,\dots,n-1$. Take $x\ge 0$ and assume that the sequence $x_0,\dots,x_n$ is given by the recurrent formula 
$x_0=x$, $x_{p+1}=G(\pi_{t_p}x_p)$  for $p=0,\dots,n-1$, 
providing that $q(x_p,t_p)>0$, and $y=x_n$.
The function $x\mapsto y$,
%\[
%y=\boldsymbol  \pi_{\mathbf  t}(x)=G(\pi_{t_{n-1}}(\dots G(\pi_{t_1}G(\pi_{t_0}x))\dots))
%\]
denoted by 
$y=\boldsymbol  \pi_{\mathbf  t}(x)$,
is called a \textit{cumulative flow} which joins   $x$ with 
$y$. 

\begin{lemma}
\label{l:com-flow}
The set  $\mathbf I$ has the following properties:
\begin{enumerate}[\rm(a)] 
\item
 If $x\in \mathbf I$ and $y=\boldsymbol  \pi_{\mathbf  t}(x)$,
 then $y\in \mathbf I$, 
\item if $x$ and $y$ are interior points of $\mathbf I$, then there exists
a cumulative flow which joins $x$ with $y$. 
\end{enumerate}
\end{lemma}
\begin{proof}
The first property is obvious, but the second one requires some justification.   
Since $q(x,a)>0$ for $a>\da(x)$, there exists $\delta>0$ such that 
$(0,\delta)\subseteq \{\pi_ax\colon q(x,a)>0\}$ and $G(\delta)\ne G(0)$. 
Then after the first jump we can join the point $x$ with each point of one of the intervals $(G(0),G(\delta))$ or $(G(\delta),G(0))$. 
If $y$ is an interior point of $\mathbf I$, then there exists $n\ge 1$ such that $y$ is also an interior point of the interval  $I_n$.
Thus we find $\rho>0$ such that we can join $x$ by a cumulative flow with each point of one of the sets
$A=(y,y+\rho)$ or $A=(y-\rho,y)$. Consider the case $A=(y,y+\rho)$. The case $A=(y-\rho,y)$ is similar.
From the continuity of the cumulative flow it follows that there is a neighbourhood $V$ of $y$ 
such that for each point $y'\in V$ we can join $x$ with each point of the set $(y',y'+\rho/2)$.
But $y\in (y',y'+\rho/2)$ if $y-\rho/2<y'<y$ which proves that we can join $x$ with $y$.
\end{proof}
  
%\item[(b)] if $x$ is an interior point of $\mathbf I$, then there exists
%a cumulative flow which joins $G(0)$ with $x$. 
%\item[(c)] for each point $x\in\mathbf I$ there
%exist $\varepsilon>0$ and  a cumulative flow which joins $x$ with each point of the interval $P_{\varepsilon}$,
%where $P_{\varepsilon}=(G(0),G(0)+\varepsilon)$
%if $G$ is increasing in a neighbourhood of zero  or $P_{\varepsilon}=(G(0)-\varepsilon,G(0))$ 
%otherwise.

%\item[(c)] for each $\varepsilon>0$ and each point $x\in\mathbf I$ there 
%exists  a cumulative flow which joins $x$ with some point from the set 
%$(G(0)-\varepsilon,G(0)+\varepsilon)$.
%\end{itemize}
%\vskip5mm

In particular, from (a) it follows that if $f\in L^1(Y)$ and $f(x_b,a)=0$  for 
$x_b\notin \mathbf I$ and $a\ge 0$, then $P(t)f(x_b,a)=0$  for 
$x_b\notin \mathbf I$ and $a\ge 0$. Thus we can restrict the semigroup  
$\{P(t)\}_{t\ge0}$ to the space $L^1(X)$ and from now on 
 by $\{P(t)\}_{t\ge0}$ we denote this restriction.

\section{Asymptotic behaviour of stochastic semigroups}
 \label{s:gen-asympt}
We now recall some general results on asymptotic stability and sweeping of 
stochastic semigroups which we will use in the next sections. 

Let  $(X,\Sigma,\mu)$ be a $\sigma$-finite measure space, 
$D$ be the set of densities and    
$\{P(t)\}_{t\ge0}$ be a  stochastic semigroup on $L^1=L^1(X,\Sigma,\mu)$.
Since the iterates of stochastic  operators also
form a (discrete time) 
semigroup we use notation $P(t)=P^t$ for their  powers 
and we formulate most of definitions and results for both types of semigroups
without distinguishing between them.

The semigroup $\{P(t)\}_{t\ge 0}$ is called
\textit{asymptotically stable} if  there exists a density $f_*$   such that
\begin{equation}
\label{d:as}
\lim _{t\to\infty}\|P(t)f-f_*\|=0 \quad \text{for}\quad f\in D.
\end{equation}
From (\ref{d:as}) it follows immediately that  $f_*$ is {\it invariant\,} with respect to
the semigroup $\{P(t)\}_{t\ge 0}$, i.e.  $P(t)f_*=f_*$ for
each $t\ge 0$.
A stochastic semigroup $\{P(t)\}_{t\ge 0}$ is
called \textit{sweeping}
from a set $B\in\Sigma$ if 
\begin{equation*}
\lim_{t\to\infty}\int_B P(t)f(x)\,\mu(dx)=0
\end{equation*}
for every  $f\in D$. 

In order to formulate a theorem on asymptotic stability of stochastic semigroups we need to introduce an auxiliary notion.

A stochastic semigroup $\{P(t)\}_{t\ge 0}$
is called \textit{partially integral} if there exists a measurable
function $k\colon (0,\infty)\times X\times X\to[0,\infty]$, called a
{\it kernel},\index{kernel} such that
\begin{equation*}
P(t)f(y)\ge\int_X k(t,x,y)f(x)\,\mu(dx)
\end{equation*}
for every density $f$ and
\begin{equation*}
\int_X\int_X  k(t,x,y)\,\mu(dy)\,\mu(dx)>0
\end{equation*}
for some $t>0$.

\begin{theorem}[\cite{PR-jmaa2}]
\label{asym-th2}
Let $\{P(t)\}_{t\ge 0}$ be a continuous time partially integral stochastic
semigroup. Assume that the  semigroup $\{P(t)\}_{t\ge 0}$ has
a unique invariant density $f_*$. If $f_*>0$ a.e., then the semigroup
$\{P(t)\}_{t\ge 0}$ is asymptotically stable.
\end{theorem}

It should be underlined that if each invariant density is positive, then an invariant density is unique or does not exist.    
Indeed, if a stochastic semigroup has two different invariant  densities $f_1$ and $f_2$,
then the function $h=(f_1-f_2)^+/ \|(f_1-f_2)^+ \|$ is also an invariant density and $h(x)=0$ on a set of positive measure $\mu$.

New results concerning positive operators on Banach lattices similar in spirit to Theorem~\ref{asym-th2} may be found in \cite{Gerlach-Gluck1,Martin-Gluck2}.

In order to formulate some result concerning  sweeping property we assume additionally that  $(X,\rho)$ is a separable
metric space and $\Sigma=\mathcal B(X)$ is the $\sigma$-algebra of Borel subsets of $X$. We also assume that ${\rm P}(t,x,A)$ is 
the \textit{transition probability function} for the semigroup 
$\{P(t)\}_{t\ge0}$, i.e. ${\rm P}(t,x,A)=P^*(t)\mathbf 1_A(x)$ for $t\ge 0$,
$x\in X$, $A\in\Sigma$. In particular if the semigroup $\{P(t)\}_{t\ge 0}$
is generated by a Markov process $(\eta_t)_{t\ge 0}$, then  
${\rm P}(t,x,A)$ is the transition probability function of this process.
We also use the notation  ${\rm P}(t,x,A)$ for stochastic operators, but now $t\in \mathbb N_+$, and instead of ${\rm P}(1,x,A)$ we briefly write   ${\rm P}(x,A)$.

We say that a stochastic  semigroup $\{P(t)\}_{t\ge0}$
satisfies condition (K) at a point $x_0\in X$ if
there exist  an $\varepsilon >0$,  a $t>0$,
and a measurable function
$\chi\ge 0$ such that $\int \chi(y)\, \mu(dy)>0$ and
\begin{equation}
\label{w-eta3}
{\rm P}(t,x,dy)\ge \chi(y)\,\mu(dy) \quad \textrm{for $x\in B(x_0,\varepsilon)$}.
\end{equation}
If a stochastic  semigroup $\{P(t)\}_{t\ge0}$ satisfies 
condition (K) at each point $x_0\in X$, then we briefly say
that this semigroup satisfies condition (K).
It is clear that if a stochastic semigroup satisfies condition (K) 
at least at one point, then this semigroup is partially integral.

\begin{theorem}[\cite{PR-JMMA2016}]
\label{col-sw}
If a (continuous or discrete time) stochastic semigroup $\{P(t)\}_{t\ge 0}$   satisfies condition {\rm (K)} and has no invariant densities, then
$\{P(t)\}_{t\ge 0}$ is sweeping from compact sets.
\end{theorem}

It is worth  mentioning that condition (K) allows us to formulate 
a general theorem on the asymptotic decomposition of stochastic and substochastic
semigroups \cite{PR-JMMA2016,PR-SD2017} from which the simple conclusion is 
Theorem~\ref{col-sw}.

The aim of the paper is to prove some results concerning asymptotic 
stability of the semigroup $\{P(t)\}_{t\ge 0}$ introduced in Section~\ref{s:restr-semigroup}.
We will proceed with the following scheme.
First we check  whether the semigroup satisfies condition (K), in particular 
$\{P(t)\}_{t\ge 0}$ is partially integral. In order to prove its stability we will apply Theorem~\ref{asym-th2}. Thus we need to check whether it has a unique invariant density $f_*$ and $f_*>0$ a.e. We show that $f_*(x_b,a)=h_*(x_b)\Phi(x_b,a)$,
where $h_*\in L^1(\mathbf I)$ is a positive fixed point of some stochastic operator $T$ on $L^1(\mathbf I)$.
We show that the operator $T$ satisfies condition (K)
and that any invariant density of this operator is strictly positive. 
Thus we have two cases. The first:  $T$ has an invariant density. Then
the semigroup $\{P(t)\}_{t\ge 0}$ has a unique invariant density $f_*>0$ a.e., and consequently it is asymptotically stable.
The second: $T$ has no invariant density. Then $T$ is sweeping from compact sets.  Observe that if the set $\mathbf I$ is bounded, then    
$\mathbf I$ is a compact set, and the operator $T$ cannot be sweeping from compact sets. Hence, if the set $\mathbf I$ is bounded, the semigroup 
$\{P(t)\}_{t\ge 0}$ is asymptotically stable.
In the case when the set $\mathbf I$ is unbounded we add some additional assumptions concerning
 our model that exclude sweeping property of $T$, so $T$ has an invariant density, which again implies 
 asymptotic stability of $\{P(t)\}_{t\ge 0}$.

\section{Condition (K)}
 \label{s:condition-K}
We now return to the semigroup $\{P(t)\}_{t\ge0}$ introduced in Section~\ref{s:restr-semigroup}
and we will prove that it satisfies condition (K).

\begin{lemma}
\label{lemma-K-lok}
Assume {\rm (A1)--(A5), (B1)},  and 
that there exist constants $t^0_1,t^0_2>0$, $x_b^0\in \mathbf I$, and $a^0\ge 0$  
such that
$q(x_b^0,a^0+t^0_1)>0$, 
$q(x_b^1,t^0_2)>0$, where $x_b^1=G(\bar x)$ and $\bar x=\pi_{a^0+t^0_1}x^0_b$,
and 
\begin{equation}
\label{nier-K}
F(\bar x)G'(\bar x)\ne F(G(\bar x)),\quad G'(\pi_{t^0_2}G(\bar x))\ne 0.
\end{equation}
Then the semigroup $\{P(t)\}_{t\ge 0}$ satisfies 
condition {\rm (K)} at the point $(x_b^0,a^0)$.
\end{lemma}
\begin{proof}
Let $(\eta_t)_{t\ge0}$ be the process defined in Section~\ref{intro}
such that $\eta_0=(x_b,a)$
and assume that this process  
until time $t>t_1+t_2$
has exactly two jumps at moments
$t_1$, $t_1+t_2$.
Then
\[
\eta_{t_1^-}=(x_b,a+t_1),\quad
\eta_{t_1}=(G(\pi_{a+t_1}x_b),0),\quad
\eta_{t_2^-}=(G(\pi_{a+t_1}x_b),t_2),
 \]
 \[
 \eta_{t_2}=(G(\pi_{t_2}G(\pi_{a+t_1}x_b)),0),\quad
 \eta_{t}=(G(\pi_{t_2}G(\pi_{a+t_1}x_b)),t-t_1-t_2).
  \]
 From continuity of the functions $q,F,G'$ we can choose $\delta>0$ and $\varepsilon>0$ such that 
\begin{gather}
\label{nier-q}
 q(x_b,a+t_1)\ge \varepsilon,\quad
 q(G(\pi_{a+t_1}x_b),t_2)\ge \varepsilon,\\
 \label{nier-FG}
 F(x)G'(x)\ne F(G(x)),\quad G'(y)\ne 0
 \end{gather}
 for $a,t_1,t_2\ge 0$, $x_b\in \mathbf I$,
$|a-a^0|<\delta$, $|x_b-x_b^0|<\delta$,
$|t_1-t_1^0|<\delta$, $|t_2-t_2^0|<\delta$, 
$x=\pi_{a+t_1}x_b$, and $y=\pi_{t_2}G(\pi_{a+t_1}x_b)$. Fix $x_b$ and $a$. Consider a function 
$\Theta$ defined on 
the rectangle 
$R=(t_1^0-\delta,t_1^0+\delta)\times (t_2^0-\delta,t_2^0+\delta)$   
by 
\[
\Theta(t_1,t_2)=(G(\pi_{t_2}G(\pi_{a+t_1}x_b)),t-t_1-t_2).
\]
Then we have 
\[
{\rm P}(t,(x_b,a),\Theta(B))\ge \varepsilon^2 |B| 
\]
for any measurable subset $B$ of $R$, where $|B|$ is the Lebesgue measure of $B$. It remains to check that $\det\Theta'(t_1,t_2)\ne 0$
for $(t_1,t_2)\in R$,
 because then
there exists $c>0$ such that 
${\rm P}(t,(x_b,a),A)\ge c|A|$ for $A\subset \Theta(R)$
and condition \eqref{w-eta3} holds with $\chi=c\mathbf 1_{\Theta(R)}$.

Since $\frac{\partial \Theta_2}{\partial t_1}
=\frac{\partial \Theta_2}{\partial t_2}=-1$ we have 
$\det\Theta'(t_1,t_2)= \frac{\partial \Theta_1}{\partial t_2}
-\frac{\partial \Theta_1}{\partial t_1}$. 
Thus, we need to check that 
$\frac{\partial \Theta_1}{\partial t_2}\ne 
\frac{\partial \Theta_1}{\partial t_1}$.
We have
\begin{align*}
\frac{\partial \Theta_1}{\partial t_1}&=
\frac{G'(y)F(y)}{F(G(\pi_{a+t_1}x_b))}G'(\pi_{a+t_1}x_b)
F(\pi_{a+t_1}x_b),\\
\frac{\partial \Theta_1}{\partial t_2}&=
G'(y)F(y).
\end{align*}
Therefore,  
$\det\Theta'(t_1,t_2)\ne 0$ if and only if \eqref{nier-FG} holds.
\end{proof}

\begin{proposition}
\label{prop:K} Assume {\rm (A1)--(A5), (B1)}. Then the semigroup 
$\{P(t)\}_{t\ge 0}$
satisfies condition~{\rm (K)}.
\end{proposition}
 \begin{proof}
 Since $F(0)=0$, $F(G(0))\ne 0$, there exists an $\varepsilon>0$ such that 
$F(x)G'(x)\ne F(G(x))$ for $x\in [0,\varepsilon)$.
Take a point $(x_b^0,a^0)\in X$. Then from conditions (A1) and (B1) it follows that there exists $t_1^0$ such that 
$q(x_b^0,a^0+t^0_1)>0$ and  $\pi_{a^0+t^0_1}x^0_b\in [0,\varepsilon)$. Thus we have $F(\bar x)G'(\bar x)\ne F(G(\bar x))$ for $\bar x=\pi_{a^0+t^0_1}x^0_b$.
From (B1) it follows that $q(G(\bar x),a)>0$ for $a>\da(G(\bar x))$.
The inequality $G'(\pi_{t^0_2}G(\bar x))\ne 0$ for some $t^0_2>\da(G(\bar x))$ follows from the fact that the function $t\mapsto\pi_tG(\bar x)$ is strictly decreasing and the set $\{x\colon G'(x)=0\}$ has Lebesgue measure zero. 
According to Lemma~\ref{lemma-K-lok} 
the semigroup $\{P(t)\}_{t\ge 0}$ satisfies 
condition {\rm (K)} at the point $(x_b^0,a^0)$.
\end{proof}

\section{Invariant density}
\label{s:inv-dens}
In this section we assume that the functions $F$, $G$, $q$ satisfy assumptions 
(A1)--(A5) and (B1)--(B2).

\begin{lemma}
\label{l:inv-dens-P-T}
%A function $f_*$ is invariant density for the semigroup
%$\{P(t)\}_{t\ge0}$ if and only if 
%\begin{equation}
%\label{f*1}
%f_*(x_b,a)=h_*(x_b)\frac{q(x_b,a)}{p(x_b,a)}
%\end{equation}
%and $h_*$ is 
%a nonzero and positive invariant function 
%for the stochastic operator
The semigroup $\{P(t)\}_{t\ge0}$ has a unique invariant density if and only if the operator 
$T$ defined on the space $L^1(\mathbf I)$ by the formula
\begin{equation}
\label{f*2}
Th(x_b)=\int_0^{\infty}\Big(P_a\big(q(\cdot,a)h(\cdot)\big)\Big)(x_b)\,da
\end{equation}
has a unique invariant density.
\end{lemma}
\begin{proof} 
A function $f_*$ is invariant with respect to the semigroup
$\{P(t)\}_{t\ge0}$ if and only if $f_*\in \mathcal D(\mathcal A)$ and $\mathcal Af_*=0$.
A function $f_*\in D\cap \mathcal D(\mathcal A)$ satisfies the condition $\mathcal Af_*=0$ if
\[
\frac{\partial f_*}{\partial a}=-pf_*,
\]
which gives
\[
f_*(x_b,a)=f_*(x_b,0)\exp\Big(-\int_0^a p(x_b,s)\,ds\Big)=f_*(x_b,0)\Phi(x_b,a).
\]
 Define  $h_*(x_b)=f_*(x_b,0)$ and assume that $h_*\in L^1(\mathbf I)$.
Then $f_*\in E$. Indeed
\begin{align*}
\int_X f_*(x_b,a)\,dx_b\,da&=\int_{\mathbf I} h_*(x_b)\bigg(\int_0^{\infty} \Phi(x_b,a)\,da\bigg) dx_b\\
&=\int_{\mathbf I} h_*(x_b)\bigg(\int_0^{\infty}\int_a^{\infty} q(x_b,r)\,dr \,da\bigg) dx_b.
\end{align*}
According to (B2), we have
\[
\int_0^{\infty}\int_a^{\infty} q(x_b,r)\,dr \,da=\int_0^{\infty}\int_0^r q(x_b,r)\,da \,dr
=\int_0^{\infty}rq(x_b,r)\,dr\le M_1,
\]
and consequently
\[
\int_X f_*(x_b,a)\,dx_b\,da\le M_1\int_{\mathbf I} h_*(x_b)\,dx_b.
\]

Since $pf_*=ph_*\Phi=h_*q$, we have
\[
\|pf_*\|_E=\int_{\mathbf I} h_*(x_b) \int_0^{\infty} q(x_b,a)\,da\,dx_b=\int_{\mathbf I} h_*(x_b) \,dx_b<\infty.
\] 
Hence $pf_*\in E$. Moreover, if $f_*$ is an invariant density, then  $h_*\in L^1(\mathbf I)$, $h_*\ge 0$, and $h_*\ne 0$.
As $\frac{\partial f_*}{\partial a}=-pf_*$, we have also $\frac{\partial f_*}{\partial a} \in E$.

Since $pf_*=h_*q$, we have  $\mathcal Pf_*(x_b)=\int_0^{\infty}\Big(P_a\big(h_*(\cdot)q(\cdot,a)\big)\Big)(x_b)\,da$.
Therefore, $f_*$ is an invariant density  if and only if 
$h_*=Th_*$, where the operator $T$ is given by $\eqref{f*2}$.
Thus, the problem of existence of an invariant density
with respect to the semigroup $\{P(t)\}_{t\ge0}$
reduces to the proof that the operator $T$ has a nonzero and positive invariant function $h_*$ in the space $L^1(\mathbf I)$.
The function $\bar h=h_*/\|h_*\|_{L^1(\mathbf I)}$  is an invariant density for $T$.  
\end{proof}

Observe that  $T$ is a stochastic operator on the space $L^1(\mathbf I)$:
\begin{align*}
\int_{\mathbf I}Th(x_b)\,dx_b&=\int_{\mathbf I}\int_0^{\infty}\Big(P_a\big(q(\cdot,a)h(\cdot)\big)\Big)(x_b)\,da\,dx_b\\
&=\int_0^{\infty}\int_{\mathbf I}\Big(P_a\big(q(\cdot,a)h(\cdot)\big)\Big)(x_b)\,dx_b\,da\\
&=\int_0^{\infty}\int_{\mathbf I}q(x_b,a)h(x_b)\,dx_b\,da=\int_{\mathbf I}h(x_b)\,dx_b.
\end{align*}
The adjoint operator of $T$ is given by the formula
\begin{equation}
\label{def-T*}
T^*f(x_b)=\int_0^{\infty}q(x_b,a)P_a^*f(x_b)\,da=\int_0^{\infty}q(x_b,a)f(G(\pi_ax_b))\,da. 
\end{equation}

\begin{lemma}
\label{l:K-for-T}
The operator $T$ satisfies condition ${\rm (K)}$.
\end{lemma}

\begin{proof}
Fix a point $x_b^0\in \mathbf I$. Then there exists $a^0>0$ such that
$q(x_b^0,a^0)>0$ and $G'(\pi_{a^0}x_b^0)\ne 0$. Let $y^0=G(\pi_{a^0}x_b^0)$.  Then we find sufficiently small $\delta\in (0,a^0)$ such that
$q(x_b,a)\ge \delta$  and
\begin{equation}
\label{in-da}
\delta\le\bigg|\dfrac{d}{da}G(\pi_{a}x_b)\bigg| \le\delta^{-1} 
\end{equation}
for all $(x_b,a)$ such that $x_b\in B(x_b^0,\delta)$ and $a\in [a^0-\delta,a^0+\delta]$. Then
\begin{align*}
{\rm P}(x_b,A)&=T^*\mathbf 1_A(x_b)=\int_0^{\infty}q(x_b,a)\mathbf 1_A(G(\pi_ax_b))\,da\\
&\ge \delta\int_{a^0-\delta}^{a^0+\delta}\mathbf 1_A(G(\pi_ax_b))\,da
\ge \delta^2\int_{y(a^0)-\delta^2}^{y(a^0)+\delta^2}\mathbf 1_A(y)\,dy
\end{align*}
for $x_b\in B(x_b^0,\delta)$ and $a\in [a^0-\delta,a^0+\delta]$, where $y(a)=G(\pi_ax_b)$.
The last inequality in the above formula follows from \eqref{in-da}.
Indeed, since $|y'(a)|\le \delta^{-1}$ we have $|da/dy|\ge \delta$. The inequality   
$\delta\le |y'(a)|$ implies $|y(a)-y(a^0)|\ge \delta |a-a^0|$, and
consequently 
$[y(a^0)-\delta^2,y(a^0)+\delta^2]\subset y\big([a^0-\delta,a^0+\delta]\big)$.
We now fix $\varepsilon\in (0,\delta)$ such that  
\[
|G(\pi_{a^0}x_b)-G(\pi_{a^0}x_b^0)|\le \delta^2/2
\]
for $x_b\in B(x_b^0,\varepsilon)$. Let $B=(y^0-\delta^2/2,y^0+\delta^2/2)$. 
Then 
\[
B\subseteq [y(a^0)-\delta^2,y(a^0)+\delta^2]
\]
and we have ${\rm P}(x_b,A)\ge \delta^2|A\cap B|$ for $x_b\in B(x_b^0,\varepsilon)$.
Thus condition (K) holds with $\chi=\delta^2\mathbf 1_B$.  
\end{proof}

\begin{lemma}
\label{l:uniq-inv-dens}
If $h_*$ is an invariant density of the operator $T$, then $h_*>0$ a.e.
\end{lemma}
\begin{proof}
Suppose, contrary to our claim, that there exists a measurable set $A$  such that $|A|>0$ and
$h_*(x)=0$ for $x\in A$. Then 
\[
\langle T^{*n}\mathbf 1_A,h_*\rangle=\langle \mathbf 1_A,T^nh_*\rangle=\langle \mathbf 1_A,h_*\rangle=0.
\]
We show that 
$L_A(x)=\sum_{n=0}^{\infty} T^{*n}\mathbf 1_A(x)>0$
for each interior point $x$ of $\mathbf I$, which leads to 
a contradiction and completes the proof.
Substituting $y=G(\pi_ax_b)$ into  \eqref{def-T*}
we check that $T^*$ is an integral operator, i.e. there exists
a measurable function $k_1\colon \mathbf I\times \mathbf I\to [0,\infty]$ such that
$T^*f(x)=\int_{\mathbf I} k_1(x,y)f(y)\,dy$. 

Observe that if 
$q(x,a_0)>0$, $G'(\pi_{a_0} x)\ne 0$  and 
$y=G(\pi_{a_0} x)$, then 
$k_1(x,y)>0$. Indeed, there exist an open interval $\Delta_1$ and $\varepsilon>0$ such that $a_0\in \Delta_1$, $q(x,a)\ge \varepsilon$ 
and $G'(\pi_ax)\ne 0$ 
for $a\in \Delta_1$.
Let $f\in L^{\infty}(\mathbf I)$ and $f\ge 0$. Then   
\[
T^*f(x)\ge \int_{\Delta_1}q(x,a)f(G(\pi_ax))\,da
\ge \varepsilon \int_{\Delta_1}f(G(\pi_ax)) \,da.
\]
Let $\Delta_2=\{G(\pi_{a}x)\colon a\in \Delta_1\}$.
We substitute $r=G(\pi_ax)$ to the last integral. 
Then
\[
dr=G'(G^{-1}(r))F(G^{-1}(r))\,da
\]
and we obtain 
\[
T^*f(x)\ge  \int_{\Delta_2}\frac{\varepsilon f(r)}{|G'(G^{-1}(r))F(G^{-1}(r))|}  \,dr.
\]
Since $y\in \Delta_2$ we have 
\[
k_1(x,y)\ge \frac{\varepsilon }{|G'(G^{-1}(y))F(G^{-1}(y))|}>0.
\]
Thus $k_1(x,y)>0$ 
for almost every point $y$ connected to $x$ in one step using the cumulative  flow.
Analogously the operator $T^{*n}$ is an integral operator with some kernel $k_n(x,y)$. We have $k_n(x,y)>0$ 
for almost every point $y$ connected to $x$ in $n$ steps.
Let $k(x,y)=\sum_{n=1}^{\infty}k_n(x,y)$ then according to condition (b) of Lemma~\ref{l:com-flow} 
we have $k(x,y)>0$ for almost every point $y$.
Thus $L_A(x)=\int_A k(x,y)\,dy>0$ for each interior point $x$ of $\mathbf I$.
\end{proof}

\section{Asymptotic stability}
\label{s:as-stab}
As we argued  in Section~\ref{s:gen-asympt} the following result holds.
\begin{theorem}
\label{t:as-stab-bounded}
Assume {\rm (A1)--(A5), (B1)--(B2)} and that $\mathbf I$ is a bounded set.
Then the semigroup $\{P(t)\}_{t\ge 0}$ is asymptotically stable.
\end{theorem}

We now consider the case when $\mathbf I$ is an unbounded set. Then $\mathbf I=[x_{\rm min},\infty)$.
The immune status is roughly proportional to the concentration of antibodies  
and their degradation rate is almost constant. It means that we can assume that 
\vskip1mm
\noindent (C1)  $\,\,\lim\limits_{x\to\infty} F(x)=-\infty$. 
\vskip1mm
It is reasonable to  assume that 
the increase of the concentration of antibodies after the infection
is bounded, i.e. 
\vskip1mm
\noindent (C2)  $\,\,G(x)\le x+L$ for some $L>0$ and all $x\in \mathbf I$. 

\begin{lemma}
\label{l:Hasm}
Under the above assumptions there exists 
$R>x_{\rm min}$  such that
\begin{equation}
\limsup_{n\to\infty} \int_{x_{\rm min}}^R T^nf(x_b)\,dx_b\ge 1/3\quad \textrm{for $f\in D$},
\label{Hasm}
\end{equation}
where $D$ stands for the set of densities in $L^1(\mathbf I)$.
\end{lemma}

 \begin{proof}
Denote by $D_1$ the set of all densities with finite first moment, i.e. such that $\int_{\mathbf I} x_bf(x_b)\,dx_b<\infty$.
 Let $f\in D_1$ and $V(x_b)=x_b$. Then
 \begin{align*}
 \int_{\mathbf I} V(x_b)Tf(x_b)\,dx_b &=\int_{\mathbf I}  f(x_b)T^*V(x_b)\,dx_b\\
 &=\int_{\mathbf I} f(x_b)
\int_0^{\infty}q(x_b,a)G(\pi_ax_b)\,da\,dx_b.
  \end{align*}
As $G(x_b)\le x_b+L$ and $\int_0^{\infty}q(x_b,a)\,da=1$ we have
\[
 \int_{\mathbf I} V(x_b)Tf(x_b)\,dx_b\le 
 \int_{\mathbf I} f(x_b)
\int_0^{\infty}q(x_b,a)\pi_ax_b\,da\,dx_b+L.
  \]
Let $\varepsilon>0$ be a constant from condition (A5). Then   
\begin{align*}
\int_0^{\infty}q(x_b,a)\pi_ax_b\,da&\le \int_0^{\varepsilon}q(x_b,a)x_b\,da+\int_{\varepsilon}^{\infty}q(x_b,a)\pi_{\varepsilon}x_b\,da \\
&=x_b+(\pi_{\varepsilon}x_b-x_b)\int_{\varepsilon}^{\infty}q(x_b,a)\,da\\
&\le x_b+\varepsilon(\pi_{\varepsilon}x_b-x_b).
\end{align*}  
Since
$\lim_{x\to\infty} F(x)=-\infty$, there exists $R>0$ such that 
$\pi_{\varepsilon}x_b-x_b\le -3L/\varepsilon$ for $x_b\ge R$, and consequently 
\[
\int_0^{\infty}q(x_b,a)\pi_ax_b\,da\le 
\begin{cases}
x_b&\textrm{for $x_b <R$},\\
x_b-3L&\textrm{for $x_b\ge R$}.
\end{cases}
\]
Thus 
\[
 \int_{\mathbf I} V(x_b)Tf(x_b)\,dx_b\le 
 \int_{\mathbf I} V(x_b)f(x_b)\,dx_b+L
 -3L\int_R^{\infty} f(x_b)\,dx_b.
\]
From this inequality it follows that $T^nf\in D_1$ 
and 
\[
 \int_{\mathbf I} V(x_b)T^{n+1}f(x_b)\,dx_b\le 
 \int_{\mathbf I} V(x_b)T^nf(x_b)\,dx_b+L
 -3L\int_R^{\infty} T^nf(x_b)\,dx_b
\]
for $n\ge 0$.
Suppose contrary to our claim that \eqref{Hasm} does not hold. Then 
there exists $n_0=n_0(f)$ such that 
\[
\int_R^{\infty} T^nf(x_b)\,dx_b\ge 2/3
\] 
for $n\ge n_0$. In consequence  
\[
\int_{\mathbf I} V(x_b)T^{n+1}f(x_b)\,dx_b\le 
 \int_{\mathbf I} V(x_b)T^nf(x_b)\,dx_b-L
\]
which implies that the integral 
$\int_{\mathbf I} V(x_b)T^nf(x_b)\,dx_b$ is negative for sufficiently
large $n$, which is impossible. Since the set $D_1$ is dense in $D$,
\eqref{Hasm} holds for $f\in D$.
\end{proof}

\begin{theorem}
\label{t:as-stab-unbounded}
Assume {\rm (A1)--(A5), (B1)--(B2)} and {\rm  (C1)--(C2)} in the case when $\mathbf I$ is an unbounded set.
Then the semigroup $\{P(t)\}_{t\ge 0}$ is asymptotically stable.
\end{theorem}

\begin{proof}
According to Lemma~\ref{l:Hasm} the operator $T$ is not sweeping from compact sets. 
Since the operator $T$ satisfies condition (K), Theorem~\ref{col-sw} implies that the operator $T$ has an invariant density $h_*$.
From Lemma~\ref{l:uniq-inv-dens} it follows that $h_*>0$ and $h_*$ is a unique invariant density with respect to $T$. 
Thus the semigroup $\{P(t)\}_{t\ge 0}$ has a unique and positive invariant density. 
Since this semigroup also satisfies condition (K) (see Proposition~\ref{prop:K}), Theorem~\ref{asym-th2} implies that 
$\{P(t)\}_{t\ge 0}$ is asymptotically stable.
\end{proof}

\section{Sweeping}
\label{s:sweeping}
We now give an example of an operator $T$ that does not have invariant densities, which
implies sweeping of the semigroup 
$\{P(t)\}_{t\ge 0}$. 

\begin{lemma}
\label{l:sweeping}
Assume that  for some constants $b,c,\gamma>0$ and $L<1$ we have
$F(x)\ge -cx$, $G(x)\ge bx$ and 
\begin{equation}
\label{c:sweep-T}
\sup\limits_{x\in\mathbf I}\int_0^{\infty}q(x,a)b^{-\gamma}  e^{c\gamma a} da<L. 
\end{equation}
Then the operator $T$ has no invariant density.
\end{lemma}

\begin{proof}
Let $V(x)=x^{-\gamma}$.
Since $\pi_ax\ge e^{-ca}x$ and $G(x)\ge bx$, we have
\begin{align*}
T^*V(x)&=\int_0^{\infty} q(x,a)(G(\pi_ax))^{-\gamma}\,da\le \int_0^{\infty} q(x,a)b^{-\gamma} e^{c\gamma a}x^{-\gamma}\,da\\
&\le Lx^{-\gamma}=LV(x).
\end{align*}
Hence $T^{*n}V(x)\le L^n V(x)$ for each $n\ge 1$.
Suppose, on the contrary,  that $T$ has an invariant density $h_*$. 
From the last inequality we obtain 
\begin{align*}
\int_{x_{\rm min}}^{\infty} V(x)h_*(x)\,dx&= 
\int_{x_{\rm min}}^{\infty} V(x)T^nh_*(x)\,dx
=\int_{x_{\rm min}}^{\infty} T^{*n}V(x)h_*(x)\,dx\\
&\le L^n \int_{x_{\rm min}}^{\infty} V(x)h_*(x)\,dx\to 0
\end{align*}
as $n\to\infty$, which shows that $h_*$ is not a density.
\end{proof}
For example if $q(x,a)=e^{-a}$ and $b>e^c$, then inequality  \eqref{c:sweep-T} holds with $\gamma=c^{-1}- (\log b)^{-1}$.

Assume (A1)--(A5), (B1)--(B2) and that the assumptions of Lemma~\ref{l:sweeping} hold.
Then  according to Lemma~\ref{l:inv-dens-P-T} the semigroup $\{P(t)\}_{t\ge 0}$ has no invariant density.
Since  the semigroup $\{P(t)\}_{t\ge 0}$ satisfies condition  (K) (see  Proposition~\ref{prop:K}), 
according to Theorem~\ref{col-sw} this semigroup is sweeping from compact subsets of $X=\mathbf I\times [0,\infty)$. 

\begin{remark}
\label{strong-sweeping}
We suppose that if  the semigroup $\{P(t)\}_{t\ge 0}$ has no invariant density, then it satisfies the following stronger sweeping condition:
\begin{equation}
\label{swep-semi}
\lim_{t\to \infty}  \int_0^{\infty}\int_{x_{\rm min}}^R P(t)f(x_b,a)\,dx_b\,da=0
\end{equation}
for each $R>x_{\rm min}$  and for each density $f$, i.e. the sweeping from all sets of the form 
$[x_{\rm min},R]\times [0,\infty)$.
Notably, this property can be interpreted as asymptotic permanent immunity of the population. 

We sketch an informal justification of \eqref{swep-semi}.
Since $T^{*n}$ are integral operators with kernels  $k_n(x,y)$ satisfying condition 
$\sum_{n=1}^{\infty}k_n(x,y)>0$,  
the operator $T$ is a conservative or dissipative (see~\cite{Fo}). 
If $T$ is a conservative or dissipative operator and satisfies condition (K), 
then there exists 
a measurable function $h_*>0$  such that $Th_*\le h_*$
and $\int_C h_*(x_b)\,dx_b<\infty$ for every compact set $C$
(see Lemma 8 \cite{PR-JMMA2016}).
Here $Th_*$ denotes pointwise limit of the sequence $Th_n$, where $(h_n)$ is any monotonic sequence of nonnegative
functions from $L^1(\mathbf I)$ pointwise convergent to $h_*$ almost everywhere.
Then the function $f_*(x_b,a)=h_*(x_b)\Phi(x_b,a)$ is subinvariant with 
respect to the semigroup $\{P(t)\}_{t\ge 0}$, i.e. $P(t)f_*\le f_*$ for all $t\ge 0$. Moreover  
\[
\int_0^{\infty} \int_C f_*(x_b,a)\,dx_b\,da<\infty
\]
for each compact set $C$. Since the semigroup $\{P(t)\}_{t\ge 0}$ has no invariant density, 
this semigroup is sweeping from the sets 
of the form $[x_{\rm min},R]\times [0,\infty)$ (see \cite{R-b95}, Corollary 3), 
i.e. this semigroup satisfies condition  \eqref{swep-semi}. 
\end{remark}

\section{Asymptotic distribution of immune status}
%\section{Asymptotic distribution of the process $\xi_t$}
\label{s:as-one-dim}
Using the results concerning the long-term behaviour of the distributions of the 
process $(\eta_t)_{t\ge 0}$ we can determine asymptotic properties of the
process $(\xi_t)_{t\ge 0}$. We start with the process $\zeta_t=(\xi_t,a_t)$, 
where $a_t=t-t_n$ is the time since the last infection.

Let $u(t,x_b,a)$ be the density of the process $(\eta_t)_{t\ge 0}$.
Then the distribution of the process $(\zeta_t)_{t\ge 0}$ is given by a
density $v(t,x,a)$ and we have the following relation between these densities: 
\begin{equation}
\label{relacja-u-w}
\int_{x_{\rm min}}^{x_b} u(t,r,a)\,dr=\int_0^{\pi_ax_b} v(t,r,a)\,dr.
\end{equation}
Differentiating both sides of (\ref{relacja-u-w}) with respect to $x_b$ we obtain  
\begin{equation}
\label{relacja-u-w2}
u(t,x_b,a)=\frac{\partial (\pi_ax_b)}{\partial x_b} v(t,\pi_a x_b,a)
 =\frac{F(\pi_ax_b)}{F(x_b)} v(t,\pi_a x_b,a).
\end{equation}
Let $u(t)(x_b,a)=u(t,x_b,a)$ and $v(t)(x,a)=v(t,x,a)$.
If the semigroup $\{P(t)\}_{t\ge 0}$ is asymptotically stable
then $\lim_{t\to \infty}\|u(t)-f_*\|=0$ in $L^1(X)$ and from
\eqref{relacja-u-w2} it follows that
\begin{equation}
\label{relacja-u-w3}
\lim_{t\to \infty}\int_X \bigg|
\frac{F(\pi_ax_b)}{F(x_b)} v(t,\pi_a x_b,a)-f_*(x_b,a)\bigg|\,dx_b\,da=0.
\end{equation}
Let 
\[
g_*(x,a)=\frac{F(\pi_{-a}x)}{F(x)} f_*(\pi_{-a} x,a)\quad \textrm{for $x\in \pi_a(\mathbf I)$}.
\]
Substituting $x=\pi_ax_b$ in \eqref{relacja-u-w3} we obtain 
\begin{equation}
\label{relacja-u-w4}
\lim_{t\to \infty}\int_0^{\infty} \int_{\pi_a(\mathbf I)} \bigg|
 v(t,x,a)-g_*(x,a)\bigg|\,dx\,da=0.
\end{equation}
Thus $\lim_{t\to \infty}\|v(t)-g_*\|=0$ in $L^1(\widetilde X)$,
where 
\[
\widetilde X=\{(x,a)\colon a\in[0,\infty),\,\, x\in \pi_a(\mathbf I)\}
\]
and $g_*$ is a unique stationary density of the process $(\zeta_t)_{t\ge 0}$ with values in $\widetilde X$.   
Moreover, from  \eqref{relacja-u-w4} it follows that the densities $f_t(x)$ of the distributions of the process $(\xi_t)_{t\ge 0}$ converge to 
$\tilde g_*(x)=\int_0^{\infty} g_*(x,a)\,da$ in the space $L^1[0,x_{\rm max})$, where $x_{\rm max}=\sup \mathbf I$.
 
\section{Examples and final remarks}
\label{s:examples}
We now consider  versions of the model with specific functions $F$, $G$ and $q$. These are models in which: 
immunity decreases exponentially; with constant increase of antibodies after infection; 
with re-infection dependent on the concentration of antibodies; and with seasonal infections.

\begin{example}[Exponential decay model]
\label{ex:F}
We consider a model with constant degradation rate of antibodies.
In such a model $F(x)=-cx$, for some $c>0$, i.e. 
the decrease of antibodies is proportional to their total number.
In this case $P_{\pi_a}f(x)=f(e^{ca}x)e^{ca}$
and the operator $P_a$ is given by the formula
\[
P_af(x)=\sum_{i\in I_x} e^{ca}f(e^{ca}\varphi_i(x))|\varphi_i'(x)|,
\]
where the functions $\varphi_i$ were introduced in Section~\ref{s:model}. 
\end{example}

\begin{example}[Constant boost of the immune level]
\label{ex:G}
In such a model $G(x)=x+K$ for some $K>0$. 
Then 
\[
P_Gf(x)=\begin{cases} 
0&\text{if $x< K$},\\
f(x-K)&\text{if $x\ge K$}.
\end{cases}
\]
If we additionally assume that  $F(x)=-cx$, then
$P_af(x)=e^{ca}f(e^{ca}(x-K))$ for $x\ge K$ and 
$P_af(x)=0$ for $x<K$.
\end{example}

The risk of re-infection depends on the concentration of antibodies
and we can consider models based on this assumption.

\begin{example}[Model with a threshold concentration of antibodies]  
\label{ex:theshold}
Usually, infection can only occur when antibody level falls below a certain threshold value $x_{{\rm th}}$. Assume that $x_{{\rm th}}\le \min G(x)$
and denote by $a_{{\rm th}}(x_b)$ 
the time after which the immune status reaches the level $x_{{\rm th}}$.
Then 
\[
a_{{\rm th}}(x_b)=\int_{x_b}^{x_{{\rm th}}}\frac{dx}{F(x)}.
\]
We assume that  $\tau=a_{{\rm th}}(x_b)+\vartheta $ is the length of the period between infections,
where $\vartheta $ is a nonnegative random variable having density $h(a)$.
Then $q(x_b,a)=0$ for $a<a_{{\rm th}}(x_b)$ and 
$q(x_b,a)=h(a-a_{{\rm th}}(x_b))$ for $a\ge a_{{\rm th}}(x_b)$.
In particular if $F(x)=-cx$, then 
\[
a_{{\rm th}}(x_b)=c^{-1}\ln(x_b/x_{{\rm th}}),\quad
q(x_b,a)=h(a-c^{-1}\ln(x_b/x_{{\rm th}})).
\]
If $m_{\vartheta}$ and  $\sigma_{\vartheta}$ are the expected value and the standard deviation  of  $\vartheta$,  
then a random variable $\tau$ with density $a\mapsto q(x_b,a)$ has
the expected value $M(x_b)= m_{\vartheta}+a_{{\rm th}}(x_b)$ and the standard deviation $\sigma_{\vartheta}$. 
\end{example}

\begin{example}[Re-infection dependent on $x^{-k}(a)$]  
\label{ex:concentr-antib}
We now consider a model in which probability of re-infection depends on the concentration of antibodies.
We recall that the rate of re-infection at time $a$ is given by
$p(x_b,a)$. Let $x(a)$ be the immune status
after the time $a$ that has elapsed since infection 
and assume that 
there exist positive constants $\kappa$ and $k$ such that
$p(x_b,a)=\kappa x^{-k}(a)$.
Since 
\[
q(x_b,a)=p(x_b,a)\exp\big(-\textstyle{\int_0^a} p(x_b,r)\,dr\big)
\]
we have
\[
q(x_b,a)=\kappa x^{-k}(a)\exp\big(-\textstyle{\int_0^a} \kappa x^{-k}(r)\,dr\big).
\]
In particular if $F(x)=-cx$, then $x(a)=x_be^{-ca}$ and we obtain
\[
q(x_b,a)=\kappa x_b^{-k}e^{cka}\exp\bigg(\frac{\kappa x_b^{-k}}{ck}\big(1-e^{cka}\big) \bigg)
\]
and 
\[
\Phi(x_b,a)=\exp\bigg(\frac{\kappa x_b^{-k}}{ck}\big(1-e^{cka}\big) \bigg).
\]
%Let $M_n(x_b)$ be the $n$-th moment of $q(x_b,a)$, i.e. 
%\[
%M_n(x_b)=\int_0^{\infty}a^n q(x_b,a)\,da.
%\]
%Then 
%\[
%M_n(x_b)=\frac{Ke^K}{(ck)^n} \int_0^{\infty} y^n e^y e^{-Ke^y}\,dy=\frac{ne^K}{(ck)^n} \int_0^{\infty} y^{n-1}  e^{-Ke^y}\,dy,
%\]
%where $K=\kappa x_b^{-k}/(ck)$.
The graph of the function $a\to q(x_b,a)$ for some values of $c,k,\kappa,x_b$ is presented in Fig.~\ref{r:im-1}.
\begin{figure}
\centerline{
\begin{picture}(350,140)(-5,-10)
\put(0,0){\includegraphics{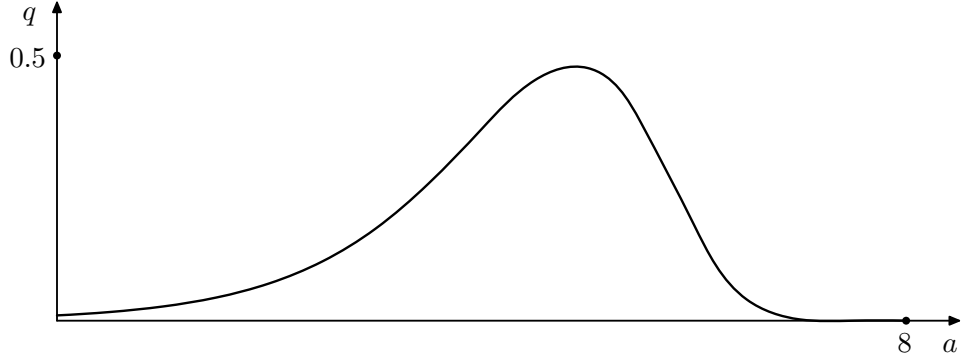}}
\put(337,-10){$a$}
\put(-11,116){$q$}
\put(320,-10){$8$}
\put(-16,98){$0.5$}
\end{picture}
}
\caption{Graph of the function $a\mapsto q(x_b,a)$ for $ck=1$ and $x_b^k=100\kappa$.}
\label{r:im-1}
\end{figure}
Let $M(x_b)$ be the expected value of $q(x_b,a)$.
Then 
\[
M(x_b)=\int_0^{\infty}a q(x_b,a)\,da=\int_0^{\infty}\Phi(x_b,a)\,da.
\]
Substituting $y=Ke^{cka}$ to the last integral, where $K=\kappa x_b^{-k}/(ck)$, 
we obtain
\[
M(x_b)=
\frac{e^K}{ck} \int_K^{\infty}y^{-1}e^{-y}\,dy=\frac{e^K}{ck}E_1(K). 
\]
Here the function $E_1$ denotes the \textit{exponential integral} $E_1(x)=\int_{x}^{\infty}y^{-1}e^{-y}\,dy $.
Since $K=\kappa x_b^{-k}/(ck)=p(x_b,0)/(ck)$ and the rate of re-infection at $a=0$ should be small, we can assume that $K$ is a small number.
It is well known that $E_1(x)=-\gamma-\ln x+o(1)$ as $x\to 0$, where $\gamma\approx 0.57721$ is Euler's constant.
Thus
\[
M(x_b)\approx  -\frac{\gamma+\ln K}{ck}=\frac{\ln x_b}{c} -  \frac{\gamma+\ln(\kappa/ck)}{ck}.
\]
Observe that the formulae for
the expected value  of the length of the period between infections $M(x_b)$
in the models from Examples~\ref{ex:theshold} and~\ref{ex:concentr-antib} are similar.
\end{example}

Certain epidemics occur seasonally and then we should include  the epidemic course in the model.
Let $T$ be the average time between their successive outbreaks.
In this case we can simply assume that $q$ can depend only on $a$, and $q$ has a distribution centered around $T$, but since an infection may occur after only a few epidemic cycles, $q$ is expected to be a multimodal function  with local maxima near points 
$T, 2T,3T,\dots$. 

\begin{example}[Re-infection dependent on the epidemic course]  
\label{ex:seasonal}
We consider a model with probability of re-infection dependent on 
the course of an epidemic and the concentration of antibodies.
We assume  that the rate of re-infection
$p(x_b,a)$ is proportional to some negative power of the immune status $x(a)$ of an individual
and the number of infected people $n(a)$ at given time $a$.
Thus 
$p(x_b,a)=\kappa x^{-k}(a)n(a)$, where 
$\kappa$ and $k$ are  positive constants.
Then 
\[
q(x_b,a)=\kappa x^{-k}(a)n(a)\exp\bigg(-\int_0^a \kappa x^{-k}(r)n(r)\,dr\bigg).
\]
In particular if $F(x)=-cx$, then $x(a)=x_be^{-ca}$ and we obtain
\[
q(x_b,a)=\kappa  x_b^{-k}e^{cka}n(a)\exp\bigg(-\int_0^a \kappa  x_b^{-k}e^{ckr}n(r)\,dr \bigg).
\]
Fig.~\ref{r:im-2} illustrates the dependence of $q$ on the course of the epidemic. 
\begin{figure}
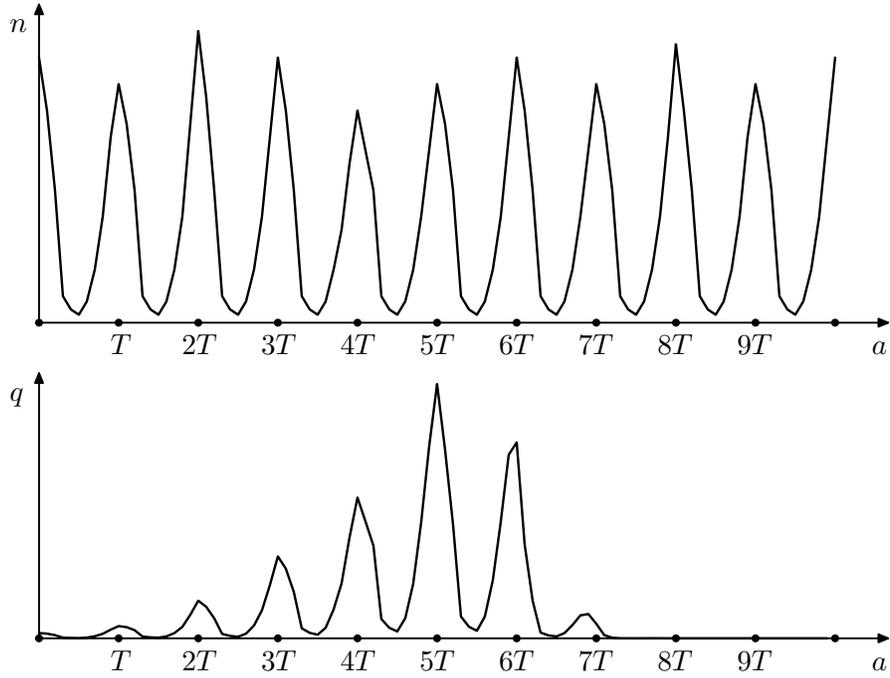

\centerline{
\begin{picture}(330,130)(0,0)
\put(0,0){\includegraphics{imm-gen-2.mps}}
\put(317,-10){$a$}
\put(-9,112){$n$}
\put(29,-10){$T$}
\put(56,-10){$2T$}
\put(86,-10){$3T$}
\put(116,-10){$4T$}
\put(146,-10){$5T$}
\put(176,-10){$6T$}
\put(206,-10){$7T$}
\put(236,-10){$8T$}
\put(266,-10){$9T$}
\end{picture}
}
\vskip1cm
\centerline{
\begin{picture}(330,100)(0,-10)
\put(0,0){\includegraphics{imm-gen-3.mps}}
\put(317,-10){$a$}
\put(-9,92){$q$}
\put(29,-10){$T$}
\put(56,-10){$2T$}
\put(86,-10){$3T$}
\put(116,-10){$4T$}
\put(146,-10){$5T$}
\put(176,-10){$6T$}
\put(206,-10){$7T$}
\put(236,-10){$8T$}
\put(266,-10){$9T$}
\end{picture}
}
\caption{Graphs of the functions $n$ and $q(x_b,\cdot)$ for $ck\approx 0.953$ and $x_b^k\bar n\approx 125\kappa$; $\bar n$ is the mean of the function $n$.}
\label{r:im-2}
\end{figure}
\end{example}

 %(see Fig.~\ref{f:multimodal})

It is worth noting that not only does the course of the epidemic affect the immune status of individuals, 
but conversely the immune status of individuals also affects the course of the epidemic. 
The  stationary density 
of the process $(\xi_t)_{t\ge 0}$  
describes the distribution of immune status in a group of people
who have undergone infections
and this distribution can be determined by serum antibody testing \cite{DEMELKER2006106,Kretzschmar2010}. 
People in this group tend to pass subsequent infections
asymptomatically but they can be potential carriers \cite{GKTD,Metcalf}. 
Thus, the heterogeneity of the population by degree of immunity should be incorporated
in epidemic modeling and vaccination planning. 
Knowing the distribution of immune status can be a good guide to determine the time after which revaccination 
should be applied to maintain population immunity
\cite{Antia,Lavine-short,Le,Wendelboe,WilsonE70}. 
Vaccination planning should also take into account the emergence of new virus variants and the waning of cross-protective immunity \cite{Reich}. 

Some viral outbreaks, such as influenza, pertussis and possibly covid, 
usually occur on an annual cycle, but subsequent outbreaks have been 
observed in longer cycles.   
This phenomenon is associated with demographic fluctuation -- a decline in the immune status
of those who have undergone infection and the emergence of new individuals. 
This is especially true for diseases occurring in infancy, such as pertussis~\cite{Wearing}. 
Including the distribution of immune status in an epidemiological model should make it easier to predict future outbreaks.

%\subsection{Acknowledgements} 

%\subsection{Bibliography}

%\begin{enumerate}[1]
%\item Use \verb"\bibliography{wileyNJD-AMA}" BST file for AMA reference style
%\item Use \verb"\bibliography{wileyNJD-APA}" BST file for APA reference style
%\item Use \verb"\bibliography{wileyNJD-AMS}" BST file for AMS reference style
%\item Use \verb"\bibliography{wileyNJD-VANCOUVER}" BST file for Vancouver reference style
%\item Use \verb"\bibliography{wileyNJD-ACS}" BST file for Chemistry reference style
%\end{enumerate}

%The normal commands for producing the reference list are:

%\begin{verbatim}

%\end{verbatim}

\end{document}